\documentclass[oneside,reqno,english]{amsart}
\usepackage[T1]{fontenc}
\usepackage[latin9]{inputenc}
\usepackage{geometry}
\usepackage{xcolor}
\usepackage{calc}
\usepackage{amsthm}
\usepackage{amssymb}

\makeatletter
\numberwithin{equation}{section}
\numberwithin{figure}{section}
\theoremstyle{definition}
\newtheorem*{defn*}{\protect\definitionname}
\theoremstyle{plain}
\newtheorem*{lem*}{\protect\lemmaname}
\theoremstyle{definition}
\newtheorem*{example*}{\protect\examplename}
\theoremstyle{plain}
\newtheorem{thm}{\protect\theoremname}
\theoremstyle{definition}
\newtheorem{defn}[thm]{\protect\definitionname}
\theoremstyle{plain}
\newtheorem{cor}[thm]{\protect\corollaryname}

\usepackage{chngcntr}
\counterwithout{equation}{section} 

\makeatother

\usepackage{babel}
\providecommand{\corollaryname}{Corollary}
\providecommand{\definitionname}{Definition}
\providecommand{\examplename}{Example}
\providecommand{\lemmaname}{Lemma}
\providecommand{\theoremname}{Theorem}

\begin{document}
\title{a magic determinant formula\\
for symmetric polynomials of eigenvalues}
\author{Jules Jacobs (julesjacobs@gmail.com)}
\begin{abstract}
Symmetric polynomials of the roots of a polynomial can be written
as polynomials of the coefficients, and by applying this theorem to
the characteristic polynomial we can write a symmetric polynomial
of the eigenvalues $a_{i}$ of an $n\times n$ matrix $A$ as a polynomial
of the entries of the matrix. We give a magic formula for this: symbolically
substitute $a\mapsto A$ in the symmetric polynomial and replace multiplication
by $\det$. For instance, for a $2\times2$ matrix $A$ with eigenvalues
$a_{1},a_{2}$,
\begin{align*}
a_{1}a_{2}^{2}+a_{1}^{2}a_{2} & =\det(A_{1},A_{2}^{2})+\det(A_{1}^{2},A_{2})
\end{align*}
 where $A_{i}^{k}$ is the $i$-th column of $A^{k}$. One may also
take negative powers, allowing us to calculate:
\begin{align*}
a_{1}a_{2}^{-1}+a_{1}^{-1}a_{2} & =\det(A_{1},A_{2}^{-1})+\det(A_{1}^{-1},A_{2})
\end{align*}
The magic method also works for multivariate symmetric polynomials
of the eigenvalues of a set of commuting matrices, e.g. for $2\times2$
matrices $A$ and $B$ with eigenvalues $a_{1},a_{2}$ and \textbf{$b_{1},b_{2}$,}
\begin{align*}
a_{1}b_{1}a_{2}^{2}+a_{1}^{2}a_{2}b_{2} & =\det(AB_{1},A_{2}^{2})+\det(A_{1}^{2},AB_{2})
\end{align*}
\end{abstract}

\maketitle

\section{Introduction}

Let $A$ be an $n\times n$ matrix with eigenvalues $a_{1},\dots,a_{n}$.
It is well known that
\begin{align*}
a_{1}+a_{2}+\dots+a_{n} & =\mathrm{tr}(A)=A_{11}+A_{22}+\dots+A_{nn}\\
a_{1}a_{2}\cdots a_{n} & =\det(A)=\sum_{\sigma\in S_{n}}(-1)^{\sigma}A_{1,\sigma(1)}A_{2,\sigma(2)}\cdots A_{n,\sigma(n)}
\end{align*}
By applying the fundamental theorem of symmetric polynomials to the
characteristic polynomial of $A$, we find that there must be such
an equation between any symmetric polynomial in the eigenvalues of
$A$ and \emph{some} polynomial in the entries of $A$. We give an
explicit formula for this polynomial in terms of determinants, which
generalises the equations above to any symmetric polynomial:
\begin{align*}
\sum_{i\in\mathbb{N}^{n}}p_{i}a_{1}^{i_{1}}\cdots a_{n}^{i_{n}} & =\sum_{i\in\mathbb{N}^{n}}p_{i}\det(A_{1}^{i_{1}},\dots,A_{n}^{i_{n}})
\end{align*}
At first, this may seem surprising, since eigenvalues are independent
of the choice of basis, whereas taking $k$-th columns of a matrix
is clearly basis dependent. Indeed, each term $p_{i}\det(A_{1}^{i_{1}},\dots,A_{n}^{i_{n}})$
is on its own not basis independent, only the whole sum is, and only
if the $p_{i}$ are coefficients of a symmetric polynomial (which
means $p_{(i_{\sigma(1)},\dots,i_{\sigma(n)})}=p_{(i_{1},\dots,i_{n})}$
for any permutation $\sigma$, or $p_{i\circ\sigma}=p_{i}$ in short,
and that only finitely many $p_{i}$ are nonzero). More generally,
\begin{align}
\sum_{i\in\mathbb{N}^{n}}p_{i}\det(A_{1}^{(i_{1})},\dots,A_{n}^{(i_{n})})\label{eq:detAj}
\end{align}
is basis independent for any family of matrices $A^{(j)}$, not necessarily
powers $A^{j}$ of a single matrix. That is, if we substitue $A^{(j)}\mapsto S^{-1}A^{(j)}S$
for some invertible matrix $S$, its value does not change. Furthermore,
if the $A^{(j)}$ commute, we have the identity 
\begin{align*}
\sum_{i\in\mathbb{N}^{n}}p_{i}a_{1}^{(i_{1})}\cdots a_{n}^{(i_{n})} & =\sum_{i\in\mathbb{N}^{n}}p_{i}\det(A_{1}^{(i_{1})},\dots,A_{n}^{(i_{n})})
\end{align*}
where $a_{k}^{(j)}$ is the $k$-th eigenvalue of $A^{(j)}$. 

Our strategy to prove this is to define a quantity that is basis independent
because its definition makes no reference to a basis, and then show
that it is equal to (\ref{eq:detAj}). Once we have shown that (\ref{eq:detAj})
is invariant under basis transformations, we pick a basis in which
the determinants become products of eigenvalues, which is possible
if the $A^{(j)}$ commute. 

By applying this identity to particular families of matrices, we get
the fundamental theorem of symmetric polynomials as a corollary, and
we are able to deduce various equations between eigenvalues and determinants,
such as those in the abstract.

\section{Preliminaries}

The proof is based on multilinear antisymmetric functions {[}1{]}.
\begin{defn*}
A function $f:V^{n}\to\mathbb{R}$ is 
\end{defn*}
\begin{itemize}
\item Multilinear if $f$ is linear in each argument, i.e. $v_{k}\mapsto f(v_{1},\dots,v_{k},\dots,v_{n})$
is linear, with $v_{i}\in V$.
\item Antisymmetric if applying a permutation $\sigma\in S_{n}$ to its
arguments multiplies it by the sign of the permutation, i.e. $f(v_{\sigma(1)},\dots,v_{\sigma(n)})=(-1)^{\sigma}f(v_{1},\dots,v_{n})$,
with $v_{i}\in V$.
\end{itemize}
Multilinear antisymmetric functions form a vector space.
\begin{defn*}
Let $\bigwedge^{n}V^{*}\subset V^{n}\to\mathbb{R}$ be the space of
multilinear antisymmetric functions.
\end{defn*}
The only property of $\bigwedge^{n}V^{*}$ we will need is that $\dim(\bigwedge^{n}V^{*})=1$
if $n=\dim(V)$. In fact, in general $\dim(\bigwedge^{n}V^{*})={\dim(V) \choose n}$
{[}1{]}. Since it is the basis of our main theorem, we will give a
proof here.
\begin{lem*}
$\dim(\bigwedge^{n}V^{*})=1$ if $\dim(V)=n$.
\end{lem*}
\begin{proof}
Let $f\in\bigwedge^{n}V^{*}$ and $v_{1},\dots,v_{n}\in V$. Take
a basis $e_{1},\dots,e_{n}$ for $V$. Writing the $v_{i}$ in terms
of the basis, we have coefficients $a_{ij}\in\mathbb{R}$ for $i,j\in\{1\dots n\}$
such that $v_{i}=\sum_{j}a_{ij}e_{j}$. The proof proceeds by expanding
$f(v_{1},\dots,v_{n})$ by multilinearity, and then using the fact
that $f$ is zero whenever there is a duplicate argument, because
applying a permutation that swaps those two positions gives $f(\dots,w,\dots,w,\dots)=-f(\dots,w,\dots,w,\dots)$,
which implies $f(\dots,w,\dots,w,\dots)=0$. This allows us to convert
the sum over any pattern of indexing into a sum over permutations.
\begin{align*}
f(v_{1},\dots,v_{n}) & =f(\sum_{i_{1}=1}^{n}a_{1i_{1}}e_{i_{1}},\dots,\sum_{i_{n}=1}^{n}a_{1i_{n}}e_{i_{n}})\\
 & =\sum_{i_{1}=1}^{n}\dots\sum_{i_{n}=1}^{n}a_{1i_{1}}\dots a_{ni_{n}}f(e_{i_{1}},\dots,e_{i_{n}})\\
 & =\sum_{i\in\{1\dots n\}\to\{1\dots n\}}a_{1i(1)}\dots a_{ni(n)}f(e_{i(1)},\dots,e_{i(n)})\\
 & =\sum_{\sigma\in S_{n}}a_{1\sigma(1)}\dots a_{n\sigma(n)}f(e_{\sigma(1)},\dots,e_{\sigma(n)})\\
 & =\sum_{\sigma\in S_{n}}(-1)^{\sigma}a_{1\sigma(1)}\dots a_{n\sigma(n)}f(e_{1},\dots,e_{e})\\
 & =\det(a)f(e_{1},\dots,e_{n})
\end{align*}
Therefore, the value $f\in\bigwedge^{n}V^{*}$ is entirely determined
by its value on $f(e_{1},\dots,e_{n})$, and conversely, any value
chosen for $f(e_{1},\dots,e_{n})$ determines a multilinear antisymmetric
function $f\in\bigwedge^{n}V^{*}$ via the equation above, so the
space is one dimensional.
\end{proof}

\section{Proof}

Let $\mathbb{I}$ be some index set. We say that $p_{i}\in\mathbb{R}$
for $i\in\mathbb{I}^{n}$ are \emph{symmetric coefficients} if only
finitely many $p_{i}$ are nonzero, and $p_{i\circ\sigma}=p_{i}$
for all $i\in\mathbb{I}^{n}$ and permutations $\sigma\in S_{n}$.
Let $\vec{A}^{(j)}:V\to V$ for $j\in\mathbb{I}$ be a family of linear
maps on a vector space $V$ of dimension $n$.\\

\textcolor{teal}{}%
\noindent{\fboxsep 10pt\fbox{\begin{minipage}[t]{1\columnwidth - 2\fboxsep - 2\fboxrule}%
\begin{example*}
\textcolor{black}{A polynomial $p(x_{1},\dots,x_{n})$ on $n$ variables
can be written as a sum of monomials 
\begin{align*}
p(x_{1},\dots,x_{n}) & =\sum_{i\in\mathbb{N}^{n}}p_{i}x_{1}^{i_{i}}\cdots x_{n}^{i_{n}}
\end{align*}
 where $p_{i}=p_{(i_{1},i_{2},\dots,i_{n})}$ is the coefficient of
$x_{1}^{i_{1}}x_{2}^{i_{2}}\dots x_{n}^{i_{n}}$. If $p$ is a symmetric
polynomial, then $p_{i}$ are symmetric coefficients with index set
$\mathbb{I}=\mathbb{N}$. For instance, if 
\begin{align*}
p(x_{1},x_{2},x_{3}) & =x_{1}+x_{2}+x_{3}=x_{1}^{1}x_{2}^{0}x_{3}^{0}+x_{1}^{0}x_{2}^{1}x_{3}^{0}+x_{1}^{0}x_{2}^{0}x_{3}^{1}
\end{align*}
then $p_{i}=1$ for $i=(1,0,0),(0,1,0),(0,0,1)$ and $p_{i}=0$ otherwise.
Note that $p_{i\circ\sigma}=p_{i}$ for all permutations $\sigma\in S_{3}$.}
\end{example*}
\end{minipage}}}\textcolor{teal}{}\\

We now define a linear map $p(\vec{A}):\bigwedge^{n}V^{*}\to\bigwedge^{n}V^{*}$
in terms of the matrices $\vec{A}^{(j)}$ and symmetric coefficients
$p_{i}$. The key observation is that a linear map from a one-dimensional
vector space to itself is just multiplication by a scalar, so we are
in effect defining a scalar here. We shall soon see that this scalar
is precisely \ref{eq:detAj}.
\begin{defn}
\label{def:pA}If $p_{i}$ are symmetric coefficients, we define $p(\vec{A}):\bigwedge^{n}V^{*}\to\bigwedge^{n}V^{*}$
by
\begin{align*}
p(\vec{A})f(v_{1},\dots,v_{n}) & =\sum_{i\in\mathbb{I}^{n}}p_{i}f(A^{(i_{1})}v_{1},\dots,A^{(i_{n})}v_{n})
\end{align*}

This does indeed preserve antisymmetry, by change of summation variable
$i=j\circ\sigma$, and using $p_{j\circ\sigma}=p_{j}$:
\begin{align*}
p(\vec{A})f(v_{\sigma(1)},\dots,v_{\sigma(n)}) & =\sum_{i\in\mathbb{I}^{n}}p_{i}f(A^{(i_{1})}v_{\sigma(1)},\dots,A^{(i_{n})}v_{\sigma(n)})\\
 & =\sum_{j\in\mathbb{I}^{n}}p_{j\circ\sigma}f(A^{(j_{\sigma(1)})}v_{\sigma(1)},\dots,A^{(j_{\sigma(n)})}v_{\sigma(n)})\\
 & =\sum_{j\in\mathbb{I}^{n}}p_{j}(-1)^{\sigma}f(A^{(j_{1})}v_{1},\dots,A^{(j_{n})}v_{n})\\
 & =(-1)^{\sigma}p(\vec{A})f(v_{1},\dots,v_{n})
\end{align*}
\\
\end{defn}

\textcolor{teal}{}%
\noindent{\fboxsep 10pt\fbox{\begin{minipage}[t]{1\columnwidth - 2\fboxsep - 2\fboxrule}%
\begin{example*}
\textcolor{black}{Given a single matrix $A\in\mathbb{R}^{3\times3}$
and index set $\mathbb{I}=\mathbb{N}$, we can pick $\vec{A}^{(j)}=A^{j}$
to be the powers of that matrix. For the symmetric coefficients $p_{i}$
from the running example, we have
\begin{align*}
p(\vec{A})f(v_{1},v_{2},v_{3}) & =f(A^{1}v_{1},A^{0}v_{2},A^{0}v_{3})+f(A^{0}v_{1},A^{1}v_{2},A^{0}v_{3})+f(A^{0}v_{1},A^{0}v_{2},A^{1}v_{3})\\
 & =f(Av_{1},Iv_{2},Iv_{3})+f(Iv_{1},Av_{2},Iv_{3})+f(Iv_{1},Iv_{2},Av_{3})\\
 & =f(Av_{1},v_{2},v_{3})+f(v_{1},Av_{2},v_{3})+f(v_{1},v_{2},Av_{3})
\end{align*}
}
\end{example*}
\end{minipage}}}\textcolor{teal}{}\\

We use the notation $[p(\vec{A})]\in\mathbb{R}$ for the scalar corresponding
to $p(\vec{A})\in End(\bigwedge^{n}V^{*})$, so that $p(\vec{A})f=[p(\vec{A})]f$
for all $f\in\bigwedge^{n}V^{*}$. Since $p(\vec{A})$ has been defined
in terms of the $A^{(j)}$, the scalar $[p(\vec{A})]\in\mathbb{R}$
is a function of the maps $A^{(j)}:V\to V$. Since we have not used
the choice of a basis for $V$ to define $p(\vec{A})$, the scalar
$[p(\vec{A})]$ is manifestly invariant under change of basis.

Pick a basis $B:\mathbb{R}^{n}\to V$ for $V$ (with basis vectors
$b_{i}=Be_{i}$, where $e_{i}\in\mathbb{R}^{n}$ is the standard basis).
We have matrix representations $M^{(j)}=B^{-1}A^{(j)}B\in\mathbb{R}^{n\times n}$
for the $A^{(j)}:V\to V$, and we wish to calculate $[p(\vec{A})]\in\mathbb{R}$
explicitly terms of the entries of $M^{(j)}\in\mathbb{R}^{n\times n}$.
\begin{thm}
\label{thm:det-formula}$[p(\vec{A})]=\sum_{i\in\mathbb{I}^{n}}p_{i}\det(M_{1}^{(i_{1})},\dots,M_{n}^{(i_{n})})$\\
\end{thm}

\textcolor{teal}{}%
\noindent{\fboxsep 10pt\fbox{\begin{minipage}[t]{1\columnwidth - 2\fboxsep - 2\fboxrule}%
\begin{example*}
\textcolor{black}{For the running example, and $B=I$,
\begin{align*}
[p(\vec{A})] & =\det(A_{1},I_{2},I_{3})+\det(I_{1},A_{2},I_{3})+\det(I_{1},I_{2},A_{3})=A_{11}+A_{22}+A_{33}
\end{align*}
}
\end{example*}
\end{minipage}}}\textcolor{teal}{}\\

\begin{proof}
Since $p(\vec{A})$ is multiplication by a scalar $[p(\vec{A})]$,
\begin{align*}
p(\vec{A})f(v_{1},\dots,v_{n}) & =[p(\vec{A})]\cdot f(v_{1},\dots,v_{n})
\end{align*}
for all $f\in\bigwedge^{n}V^{*}$ and vectors $(v_{1},\dots,v_{n})$.
Taking $f(w_{1},\dots,w_{n})=\det(B^{-1}w_{1},\dots,B^{-1}w_{n})$
and $(v_{1},\dots,v_{n})=(Be_{1},\dots,Be_{n})$ to be the basis vectors,
on the right hand side 
\begin{align*}
f(v_{1},\dots,v_{n}) & =\det(B^{-1}Be_{1},\dots,B^{-1}Be_{n})=\det(e_{1},\dots,e_{n})=1
\end{align*}
and on the left hand side
\begin{align*}
p(\vec{A})f(v_{1},\dots,v_{n}) & =\sum_{i\in\mathbb{I}^{n}}p_{i}\det(B^{-1}A^{(i_{1})}Be_{1},\dots,B^{-1}A^{(i_{n})}Be_{n})=\sum_{i\in\mathbb{I}^{n}}p_{i}\det(M_{1}^{(i_{1})},\dots,M_{n}^{(i_{n})})
\end{align*}
giving $[p(\vec{A})]=\sum_{i\in\mathbb{I}^{n}}p_{i}\det(M_{1}^{(i_{1})},\dots,M_{n}^{(i_{n})})$.
\end{proof}
This shows that the value of $[p(\vec{A})]=\sum_{i\in\mathbb{I}^{n}}p_{i}\det(M_{1}^{(i_{1})},\dots,M_{n}^{(i_{n})})$
does not depend on the basis, since the left hand side is defined
without reference to the basis. By picking a basis in which the $A^{(j)}$
are all upper triangular, we can relate $[p(\vec{A})]$ to the eigenvalues
of the $A^{(j)}$.
\begin{thm}
\label{thm:eigs-formula}If the $A^{(j)}$ commute, then $[p(\vec{A})]=\sum_{i\in\mathbb{I}^{n}}p_{i}a_{1}^{(i_{1})}\cdots a_{n}^{(i_{n})}$,
where $a_{k}^{(j)}$ are the eigenvalues of $A^{(j)}$.\\
\\
\textcolor{teal}{}%
\noindent{\fboxsep 10pt\fbox{\begin{minipage}[t]{1\columnwidth - 2\fboxsep - 2\fboxrule}%
\begin{example*}
\textcolor{black}{For the running example, suppose we have a basis
in which $A$ is upper triangular (e.g. the Jordan basis), then 
\begin{align*}
[p(\vec{A})] & =\det(A_{1},I_{2},I_{3})+\det(I_{1},A_{2},I_{3})+\det(I_{1},I_{2},A_{3})=a_{1}+a_{2}+a_{3}
\end{align*}
 where $a_{1},a_{2},a_{3}$ are the eigenvalues of $A$.}
\end{example*}
\end{minipage}}}\textcolor{teal}{}\\
\end{thm}

\begin{proof}
Commuting matrices have a basis in which they are simultaneously upper
triangular, by Schur decomposition {[}2{]}. The diagonal of those
upper triangular matrices $M^{(j)}$ will contain the eigenvalues
$a_{k}^{(j)}$. In this case, $\det(M_{1}^{(i_{1})},\dots,M_{n}^{(i_{n})})$
is the determinant of an upper triangular matrix, with eigenvalues
$a_{k}^{(i_{k})}$ on the diagonal. Hence $\det(M_{1}^{(i_{1})},\dots,M_{n}^{(i_{n})})=a_{1}^{(i_{1})}\cdots a_{n}^{(i_{n})}$,
and substituting this into theorem (\ref{thm:det-formula}) gives
the eigenvalue formula for $p(\vec{A})$.
\end{proof}
The condition that the $A^{(j)}$ commute is not a necessary condition,
because there are upper triangular matrices that do not commute.

\section{Corollaries}

By combining theorems (\ref{thm:det-formula}) and (\ref{thm:eigs-formula})
we can justify the magic formula for converting symmetric expressions
involving eigenvalues into symmetric expressions involving determinants.
We first show this by example: let $A,B$ be commuting $3\times3$
matrices with eigenvalues $a_{1},a_{2},a_{3}$ and $b_{1},b_{2},b_{3}$.
Formally, we take the index set $\mathbb{I}=\{A,B\}$ in the theorems.
Next, we choose symmetric coefficients $p_{i}$ for $i\in\mathbb{I}^{3}$;
we must choose values for $p_{AAA},p_{AAB},p_{ABA},\dots,p_{BBB}$
that are symmetric under permutations of the indices. We choose $p_{AAB}=p_{ABA}=p_{BAA}=1$
and the rest $0$. The theorems (\ref{thm:det-formula}) and (\ref{thm:eigs-formula})
give us the equation
\begin{align*}
a_{1}a_{2}b_{3}+a_{1}b_{2}a_{3}+b_{1}a_{2}a_{3} & =\det(A_{1},A_{2},B_{3})+\det(A_{1},B_{2},A_{3})+\det(B_{1},A_{2},A_{3})
\end{align*}
We can now proceed to pick the matrices $A,B$ in this equation, as
long as they commute. For instance, given a matrix $C$, we could
pick $A=C^{3}$ and $B=C^{-1}$. Or, given commuting matrices $C,D$,
we could pick $A=C^{2}+D$ and $B=CD$. Or we could pick $A=(C+D)^{-1}$
and $B=\exp(CD)$. We thus obtain various relations between eigenvalues
and determinants by substituting, e.g. $A=(C+D)^{-1}$, $B=\exp(CD)$,
and $a_{i}=(c_{i}+d_{i})^{-1}$, $b_{i}=\exp(c_{i}d_{i})$ into the
equation.

Rather than trying to capture this general method in a theorem, we
present a few special cases.
\begin{cor}
Let $q(b_{1},\dots,b_{n})=\sum_{i\in\mathbb{N}^{n}}p_{i}b_{1}^{i_{1}}\dots b_{n}^{i_{n}}$
be a symmetric polynomial in the eigenvalues of an $n\times n$ matrix
$B$, then $q(b_{1},\dots,b_{n})=\sum_{i\in\mathbb{N}^{n}}p_{i}\det(B_{1}^{i_{1}},\dots,B_{n}^{i_{n}})$.
\end{cor}

\begin{proof}
Apply theorems (\ref{thm:det-formula}) and (\ref{thm:eigs-formula})
with index set $\mathbb{I}=\mathbb{N}$ and matrices $A^{(j)}=B^{j}$.
\end{proof}
\begin{cor}
Let $q(x)=\sum_{k=0}^{n}a_{k}x^{k}$ be a polynomial with roots $r_{k}$.
Then a symmetric polynomial in the roots $r_{k}$ can be written as
a polynomial in the coefficients $a_{k}$.
\end{cor}

\begin{proof}
Apply the previous corollary with $B$ being the companion matrix
of $q$. The eigenvalues of $B$ are the $r_{k}$. The entries of
the companion matrix are all $0$, or $1$, or $a_{k}$, so $\det(B_{1}^{i_{1}},\dots,B_{n}^{i_{n}})$
is a polynomial in the $a_{k}$.
\end{proof}
A symmetric polynomial $q$ can be seen as a function of a vector
$x\in\mathbb{R}^{n}$ satisfying $q(Px)=q(x)$ for permutation matrices
$P$.
\begin{defn}
A multivariate symmetric polynomial $q(X)$ is a polynomial function
of the entries of a matrix $X\in\mathbb{R}^{n\times m}$ satisfying
$q(PX)=q(X)$ for permutation matrices $P$.
\end{defn}

The permutation $P$ permutes the rows of $X$, but keeps each row
together. A symmetric polynomial is the special case $m=1$, when
each row consists of a single entry.
\begin{cor}
Let $B^{(j)}$ for $j\in\{1,\dots,m\}$ be commuting matrices with
eigenvalues $b_{i}^{(j)}$, and define $X_{ij}=b_{i}^{(j)}$ to be
the matrix of eigenvalues. Then a multivariate symmetric polynomial
$q(X)$ can be written as a polynomial in the entries of the $B^{(j)}$.
\end{cor}

\begin{proof}
Take $\mathbb{I}=\mathbb{N}^{m}$ and $A^{(g)}=(B^{(1)})^{g_{1}}\cdots(B^{(m)})^{g_{m}}$
in theorems (\ref{thm:det-formula}) and (\ref{thm:eigs-formula}).
\end{proof}

\section{Footnote}

The same definition (\ref{def:pA}) works for $p(\vec{A}):\bigwedge^{k}W^{*}\to\bigwedge^{k}V^{*}$
also when $W\neq V$ and $k\neq n$. We can view $p(\vec{A})f$ as
a generalised pullback of $f$ along a list of maps $A^{(j)}:V\to W$.
The ordinary pullback $A^{*}:\bigwedge^{k}W^{*}\to\bigwedge^{k}V^{*}$
is the special case of a single map $A:V\to W$. We have, in general
$p(X\vec{A}Y)=Y^{*}p(\vec{A})X^{*}$, where $X\vec{A}Y$ is simultaneous
conjugation $(X\vec{A}Y)^{(j)}=XA^{(j)}Y$. This gives us a slightly
stronger version of theorem (\ref{thm:det-formula}): if $k=\dim(V)=\dim(W)$,
then $X^{*}=\det(X)$ and $Y^{*}=\det(Y)$, so we see that if we multiply
the $M^{(j)}$ on the left by $X$ and on the right by $Y$, the value
of $[p(\vec{A})]$ gets multiplied by $\det(X)\det(Y)$. Theorem (\ref{thm:det-formula})
tells us that the value does not change if we do a basis transformation,
but here we get information about the case $X\neq Y^{-1}$. It is
also possible to generalise the proof of theorem (\ref{thm:det-formula})
directly.

\section*{References}

{[}1{]} Nicolas Bourbaki. Elements of mathematics, Algebra I. Springer-Verlag,
1989. 

{[}2{]} R.A. Horn and C.R. Johnson. Matrix Analysis. Cambridge University
Press, 1985.
\end{document}